\documentclass[12pt,twoside]{amsart}
\usepackage{hyperref,mathtools}
\usepackage{amsthm, amsmath, amscd, amssymb,centernot}
\usepackage{amsfonts}
\usepackage[all]{xy}
\usepackage{palatino,eulervm}
\usepackage[T1]{fontenc}
\usepackage{xcolor}
\usepackage[left=2.5cm,top=2.5cm,bottom=3cm,right=2.5cm]{geometry}
\usepackage{tikz-cd}

\theoremstyle{plain}
\numberwithin{equation}{section}
\newtheorem{theorem}{Theorem}[section]
\newtheorem*{theorem*}{Theorem}
\newtheorem{proposition}[theorem]{Proposition}
\newtheorem{lemma}[theorem]{Lemma}
\newtheorem{corollary}[theorem]{Corollary}

\theoremstyle{definition}

\newcommand{\sslash}{\mathbin{/\mkern-6mu/}}
\newenvironment{altenumerate}
{\begin{list}
		{\textup{(\theenumi)} }
		{\usecounter{enumi}
			\setlength{\labelwidth}{0pt}
			\setlength{\labelsep}{2pt}
			\setlength{\leftmargin}{0pt}
			\setlength{\itemsep}{\the\smallskipamount}
			\renewcommand{\theenumi}{\roman{enumi}}
	}}
	{\end{list}}
\begin{document}
	\title[A remark on a theorem of Narasimhan and Ramanan]{A remark on a theorem of Narasimhan and Ramanan}

	\author[Jagadish Pine]{Jagadish Pine}
	%\address{Chennai Mathematical Institute, H1 SIPCOT IT Park, Siruseri, Kelambakkam 603103, India}
	%\email{pine@cmi.ac.in}
	
	%\thanks{The author was partially supported by a grant from Infosys
%		Foundation. }
	
	\date{15 November, 2024}
	\maketitle
	\begin{abstract}
	In this short note, we provide an alternative proof of a notable theorem by Narasimhan and Ramanan. The theorem states that the moduli space of $S$-equivalence classes of semistable rank $2$ vector bundles over a curve $X$ of genus $2$ with trivial determinant is isomorphic to $\mathbb{P}^3$. Our proof relies on a criterion by Bauer and Szemberg, which characterizes projective spaces among smooth Fano varieties using Seshadri constants.
	\end{abstract}

	\section{Introduction}\label{intro}
	Let $X$ be an irreducible smooth, projective algebraic curve over the complex numbers $\mathbb{C}$. Throughout the note we will assume that the genus of the curve $X$ is $2$. Let $SU_{X}(2)$ be the moduli space of S-equivalence classes of semistable rank $2$ vector bundles over $X$ with trivial determinant. These are the (rank$=2$, degree is even) only instances of non-coprime rank and degree where the moduli space is nonsingular. A well-known theorem by Narasimhan and Ramanan states that the moduli space $SU_{X}(2)$ is isomorphic to the $3$-dimensional projective space $\mathbb{P}^{3}$ \cite[Theorem 2, \S 7]{MR0242185}.\\ %It is the only case of non-coprime rank and degree where the moduli space remains nonsingular. It is a well known theorem of Narasimhan-Ramanan that the moduli space $SU_{X}(2)$ is isomorphic to the $3$-dimensional projective space $\mathbb{P}^{3}$ \cite[Theorem 2, \S 7]{MR0242185}.\\
	
	 The proof proceeds through several steps as follows. Let $\text{Pic}^1(X)$ be the variety of isomorphism classes of degree $1$ line bundles on $X$. Then the natural embedding $X \hookrightarrow \text{Pic}^1(X)$ defines the $\theta$ divisor in $\text{Pic}^1(X)$. Let $L_{2\theta}$ be the line bundle associated to the diviosr $2\theta$. Any vector bundle $W$ in $SU_{X}(2)$ can be obtained as an extension in $Ext^1(\zeta, \zeta^{-1})$ for some $\zeta \in \text{Pic}^1(X)$. The space $\mathbb{P}(H^1(X, \zeta^{-2}))$ can be identified with the set of positive divisors in $\text{Pic}^1(X)$ that pass through the point $\zeta$ and are linearly equivalent to $2\theta$. Consequently, the bundle $W$ determines a positive divisor $C_W \in \mathbb{P}(H^0(\text{Pic}^1(X), L_{2\theta}))$. Since $C_W$ only depends on the S-equivalance class of $W$, this correspodence defines a map $D:SU_{X}(2) \rightarrow \mathbb{P}(H^0(\text{Pic}^1(X), L_{2\theta}))$. The proof concludes by showing that the map $D$ is bijective, thereby establishing an isomorphism.\\ %\cite[Theorem 2, \S 7]{MR0242185}.\\
	
	It may be useful to consider a new perspective on this intriguing result. In this brief note, we present an alternative proof using Seshadri constants. These constants have previously been applied in moduli problems, such as %Seshadri constants have been employed in moduli problems in the past. For instance, they have been used to
	 characterizing Jacobians of curves among principally polarized abelian varieties in lower dimensions  \cite{MR2112583}. We believe that similar techniques could prove valuable for other moduli problems as well. We begin by recalling their definition: Let $L$ be a nef line bundle on a projective variety $Y$. 
	For a point $y \in Y$, The \textit{Seshadri constant} of $L$ at $y$ is defined as
	$$\varepsilon(Y,L,y):= \inf\limits_{\substack{y \in C}} \frac{L\cdot
		C}{{\rm mult}_{y}C} \;\: ,$$
	where the infimum is taken over all irreducible and reduced curves $C \subset Y$ passing through $y$. Here 
	$L\cdot C$ denotes the intersection multiplicity, and ${\rm mult}_y C$ denotes the 
	multiplicity of the curve $C$ at $y$. If the variety is clear from the context, we denote the Seshadri constant simply by $\varepsilon(L,y)$.\\
	
	 We will rely on the following key criterion, which characterizes projective spaces among smooth Fano varieties in terms of Seshadri constants.
	%In this short note we will give a different proof of this result using the following/a criterion for a smooth Fano variety to be a projective space in terms of Seshadri constants.
	\begin{theorem}\cite[Theorem 2]{BAUER20092134}
		\label{crucialthm}
		Let $Y$ be a smooth Fano variety of dimension $n$ such that there exists a point $y \in Y$ with 
		\[
		\epsilon(-K_Y, y) = n+1
		\]
		Then $Y$ is the projective space $\mathbb{P}^n$.
	\end{theorem}
For any smooth Fano variety of dimension $n$, it is known that [\cite{BAUER20092134}, Corollary 1.3]
\[
\epsilon(-K_Y, y) \le n+1
\]  
for all $y \in Y$. 

In this note we will prove the following :
\begin{theorem}
	\label{mainthm}
	Let $x$ be any point of $SU_{X}(2)$. Then $\epsilon(-K_{SU_{X}(2)}, x) \ge 4$.
\end{theorem}
Thus as a consequence of Theorem \eqref{crucialthm} we will have $SU_{X}(2) \cong \mathbb{P}^3$.\\

There exists a generalized $\Theta$ divisor on $SU_{X}(2)$ which is the descendant of an ample line bundle arising from the GIT quotient. In $\S 2$, after providing a brief overview of the construction of % we will discuss the smoothness of 
$SU_{X}(2)$, and its key properties, we will demonstrate the base point freeness of the generalized $\Theta$ divisor. In $\S 3$ we will prove Theorem \eqref{mainthm}. Once the base point freeness of the generalized $\Theta$ divisor is established, %once we know that the generalized $\Theta$ is base point free, 
the proof of Theorem \eqref{mainthm} is straightforward. 

\section{The moduli space $SU_{X}(2)$, and the generalized theta divisor}
There exists a Quot scheme $\text{Quot}(V \otimes \mathcal{O}_X, P(t))$ which parameterizes all coherent sheaves $\mathcal{E}$ on $X$ which are quotients of $V \otimes \mathcal{O}_X$ for some vector space $V$ over $\mathbb{C}$, and the sheaf $\mathcal{E}$ has Hilbert polynomial $P(t)$ with respect to a fixed ample line bundle on $X$. Let $q : V \otimes \mathcal{O}_X \rightarrow \mathcal{E}$ be the quotient map. The scheme $\text{Quot}(V \otimes \mathcal{O}_X, P(t))$ is a fine moduli space i.e., there exists an universal coherent sheaf $\mathcal{F}$ on $X \times \text{Quot}(V \otimes \mathcal{O}_X, P(t))$ with a quotient map $V \otimes \mathcal{O}_{X \times \text{Quot}(V \otimes \mathcal{O}_X, P(t))} \rightarrow \mathcal{F}$. Let $R^{ss}$ be the open subset of $\text{Quot}(V \otimes \mathcal{O}_X, P(t))$ which contains all the semistable locally free sheaves $\mathcal{E}$ such that the induced map $H^0(q) : V \rightarrow H^0(\mathcal{E})$ is an isomorphism. The moduli space $SU_{X}(2)$ is the GIT quotient $R^{ss} \sslash SL(V)$. \\

The moduli space $SU_{X}(2)$ is a smooth, projective variety with a known dimension of $(r^2-1)\cdot (g-1)=3$. For a proof of smoothness using Luna's \'etale slice theorem we refer to [\cite{MR1418944}, Theorem 4]. We also have $\text{Pic}(SU_{X}(2)) \cong \mathbb{Z}$ \cite[Theorem B, P. 55]{MR0999313}. The generator $\Theta$ is an ample line bundle known as the generalized $\Theta$ divisor.  The canonical line bundle $K_{SU_{X}(2)}$ has the following relation with $\Theta$ \;\cite[Theorem F.]{MR0999313}
\begin{equation}
	\label{cantheta}
-K_{SU_{X}(2)} \cong \Theta^4
\end{equation} %For the rest of the section we will prove that the divisor $\Theta$ is base point free. %Due to a criterion of Beauville, the base point freeness of the generalized $\Theta$. follows from the Proposition \eqref{basepointprop}. 
%It is a known fact. %due to a lack of clear reference%
%% For completeness we will give another proof of it. %in Proposition \eqref{basepointprop}. 
% We will need the following 
 In the remainder of this section, we will demonstrate that the divisor $\Theta$ is base point free. While this is a known fact, we will provide an alternative proof for completeness. To do so, we require the following:
 %lemma due to C. P. Ramanujam
\begin{lemma}\cite[Lemma on P. 56]{mumford2008abelian}
	\label{ramanujam trick}
	Let $Z$ be an irreducible quasi projective variety over $\mathbb{C}$. Then there exists an irreducible curve passing through any two points $P, Q$ in $Z$.
\end{lemma}
%\begin{proof}
%	If $\text{dim}(Z) = 1$, then we are done. Let $\text{dim}(Z) \ge 2$. We will consider the blow up $\pi : Bl_{P, Q}Z \rightarrow Z$. Any hyperplane section coming from an embedding of $Bl_{P, Q}Z$ inside a projective space will intersect the two exceptional divisors. If the intersections are zero, then the hyperplane section is the pullback of a divisor by $\pi$ or from an intermediate blow up which is not ample. Hence this is not possible. A generic hyperplane section $H'=H \cap Bl_{P, Q}Z$ is irreducible by Bertini's theorem. Since $H'$ is of codimesion $1$, the image $\pi(H')$ is an irreducible subvariety of codimension at least $1$ containing $P, Q$. We will use induction by repeating the same process on $H'$ for two points on the two exceptional divisors.  
%\end{proof}
\begin{proposition}
	\label{basepointprop}
	Let $E$ be a semistable vector bundle of rank $2$ and trivial determinant on $X$. Then there exists a line bundle $M \in Pic^1(C)$ such that $H^{0}(M \otimes E) = 0$
\end{proposition}
\begin{proof}We employ degeneration techniques. In Case 1, we establish the result for any strictly semistable bundle. In Case 2, we extend the proof to stable bundles by degenerating them into semistable ones.
	
	\textbf{Case 1.} Let $E \in SU_{X}(2)$ be a semistable but not stable bundle. Then $E$ is S-equivalent to $\mathcal{L} \oplus \mathcal{L}^{-1}$, where $\text{deg}(\mathcal{L}) = 0$. The degree $1$ line bundles which have a nontrivial section are of the form $\mathcal{O}_{X}(x)$ for some $x \in X$. Thus they are contained inside the image of the natural map $X \rightarrow Pic^1(X)$. Since $\{\mathcal{L} \otimes M : M \in Pic^1(X)\} = Pic^1(X)$, and $\text{dim}\; Pic^1(X) =2$, we will have $H^0(\mathcal{L} \otimes M) = 0$ for all $\mathcal{L} \otimes M$ avoiding the image of $X$ in $Pic^1(X)$. Similarly for $\mathcal{L}^{-1}$ we can choose $M' \in Pic^1(X)$ in the complement of a subvariety of dimension at most $1$ such that $H^0(\mathcal{L}^{-1} \otimes M') = 0$. Thus there a common $M$ in $Pic^1(X)$ for which $H^0(\mathcal{L} \otimes M) = 0 = H^0(\mathcal{L}^{-1} \otimes M)$. Hence $H^0(X, (\mathcal{L} \oplus \mathcal{L}^{-1}) \otimes M) =0$.
	
	\textbf{Case 2.} Let $E \in SU_{X}(2)$ be a stable bundle. Let $\pi : R^{ss} \rightarrow R^{ss} \sslash SL(V) = SU_{X}(2)$ be the GIT quotient morphism. The scheme $R^{ss}$ is an irreducible smooth quasiprojective variety. For a proof of irreducibility, refer to \cite[Remark 5.5, P. 140]{MR546290}. We take a preimage of $E$ in $R^{ss}$. By an abuse of notation we will denote it by $E$. By Lemma \eqref{ramanujam trick} there exists an irreducible curve $C$ in $R^{ss}$ connecting $E$ to a semistable bundle of the form $\mathcal{L} \oplus \mathcal{L}^{-1}$. Let $\mathcal{U} \rightarrow X \times R^{ss}$ be the universal family. We restrict the universal family to $X \times C$, denoted by $\mathcal{U}_{C} \rightarrow X \times C$. So $\mathcal{U}_{C}$ is a family of locally free sheaves such that there exist $x, x_0 \in C$ with $(\mathcal{U}_{C})_x \cong E$, and $(\mathcal{U}_{C})_{x_0} \cong \mathcal{L} \oplus \mathcal{L}^{-1}$. Let $p_1$, and $p_2$ be the projections from $X \times C$ to $X$, and $C$ respectively. 
	
	% https://q.uiver.app/#q=WzAsNixbMSwwLCJcXG1hdGhjYWx7VX1fQyJdLFszLDAsIlxcbWF0aGNhbHtVfSJdLFszLDEsIlggXFx0aW1lcyBSXntzc30iXSxbMSwxLCJYIFxcdGltZXMgQyJdLFswLDIsIlgiXSxbMiwyLCJDIl0sWzAsM10sWzEsMl0sWzMsMiwiIiwxLHsic3R5bGUiOnsidGFpbCI6eyJuYW1lIjoiaG9vayIsInNpZGUiOiJ0b3AifX19XSxbMCwxLCIiLDEseyJsYWJlbF9wb3NpdGlvbiI6NjAsInN0eWxlIjp7InRhaWwiOnsibmFtZSI6Imhvb2siLCJzaWRlIjoidG9wIn19fV0sWzMsNCwicF8xIiwyXSxbMyw1LCJwXzIiXV0=
%	\[\begin{tikzcd}
%		&&&&& {} \\
%		& {\mathcal{U}_C} & {} & {\mathcal{U}} \\
%		& {X \times C} & {} & {X \times R^{ss}} \\
%		X && C &&& {}
%		\arrow[from=2-2, to=3-2]
%		\arrow[from=2-4, to=3-4]
%		\arrow[hook, from=2-2, to=2-4]
%		\arrow["{p_1}"', from=3-2, to=4-1]
%		\arrow["{p_2}", from=3-2, to=4-3]
%		\arrow[hook, from=3-2, to=3-4]
%		\arrow["\ulcorner"{anchor=center, pos=0.115}, draw=none, from=2-2, to=3-4]
%	\end{tikzcd}\]
Let $\mathcal{U}'$ be the family $\mathcal{U}_{C} \otimes p_1^* M$, where $M \in Pic^1(X)$ is the line bundle chosen in case $1$. The projection $p_2 : X \times C \rightarrow C$ is projective, and flat. The map $C \rightarrow \mathbb{Z}$ defined by $c \rightarrow \text{dim}H^0(X, \mathcal{U}'_c)$ is upper semicontinuous. At $x_0$ we have $H^0(X, \mathcal{U}'_{x_0}) \cong H^0(X, (\mathcal{L} \oplus \mathcal{L}^{-1}) \otimes M) =0$. Hence there exists an open set $V$ in $C$ containing $x_0$ such that $H^0(X, \mathcal{U}'_v) = 0$ for all $v \in V$. To complete the proof, it suffices to show that $V=C$. \\
	
	By Riemann-Roch we will have $h^0(X, M^{\otimes 2}) \ge \text{deg}(M^{\otimes 2}) + 1 - g = 2 + 1 - 2 = 1$. A section in $H^0(X, M^{\otimes 2})$ will define the short exact sequence $0 \rightarrow \mathcal{O}_{X} \rightarrow M^{\otimes 2} \rightarrow T$, where $T$ is a torsion sheaf. By  tensoring with $M$, and pulling back the short exact sequence to $X \times C$ we will get $0 \rightarrow p_1^*M \rightarrow p_1^*M^{\otimes 3} \rightarrow T'$. Tensoring with $\mathcal{U}_{C}$, and applying $(p_2)_*$ we get the injection $0 \rightarrow (p_2)_* (p_1^*M \otimes \mathcal{U}_{C}) \rightarrow (p_2)_* (p_1^*M^{\otimes 3} \otimes \mathcal{U}_{C})$. For $c \in C$, using Serre duality we have $$H^1(X, M^{\otimes 3}\otimes \mathcal{U}_c) \cong H^0(X, M^{\otimes -3}\otimes \mathcal{U}_c^{\vee} \otimes K_{X})$$ Since $(M^{\otimes -3}\otimes \mathcal{U}_c^{\vee} \otimes K_{X})$ is a semistable vector bundle of degree $-2$, it can not have any section. Thus $H^1(X, M^{\otimes 3}\otimes \mathcal{U}_c) =0$ for all $c \in C$. Hence $\text{dim} \; H^0(X, M^{\otimes 3}\otimes \mathcal{U}_c)$ remains constant. By Grauert's theorem the sheaf $(p_2)_* (p_1^*M^{\otimes 3} \otimes \mathcal{U}_C)$ is locally free. Since the sheaf $(p_2)_* (p_1^*M \otimes \mathcal{U}_C)$ restricted to the open set $V$ is zero, it is a torsion sheaf contained in $(p_2)_* (p_1^*M^{\otimes 3} \otimes \mathcal{U}_C)$. Hence $(p_2)_* (p_1^*M \otimes \mathcal{U}_C) =0$. Since the fiber of the sheaf $(p_2)_* (p_1^*M \otimes \mathcal{U}_C)$ at $x$ is $H^0(X, E \otimes M)$, we have $H^0(X, E \otimes M) =0$.

\end{proof}
%Due to a criterion of Beauville, the base point freeness of the generalized $\Theta$
The base locus of the $\Theta$-divisor is characterized by the following criterion of Beauville \cite[see \S 3]{beauville1994vector} $$\{E \in SU_{X}(2) : H^{0}(M \otimes E) \ne 0 \;\text{for all}\; M \; \text{in Pic}^1(X)\}$$ Thus we will have the following
\begin{corollary}
	\label{basepoint}
	The divisor $\Theta$ on the moduli space $SU_{X}(2)$ is base point free.
\end{corollary}
\section{proof of the main theorem}
The proof of Theorem \eqref{mainthm} is based on the following lemma :
\begin{lemma}
	\label{bezout}
	Let $x$ be a point in $SU_{X}(2)$. Let $D$ be a reduced, and irreducible projective curve contained in $SU_{X}(2)$ which passes through $x$. Then there exists an effective divisor in the linear system $\lvert\Theta \rvert$ satisfying the following : 
	\begin{altenumerate}
		\item [{ \rm(1)}] The divisor $E$ contains $x$, i.e., $x \in D \cap E$.
		\item [{ \rm(2)}] The curve $D$ is not contained in $E$.
	\end{altenumerate}
\end{lemma}
\begin{proof}
Since $\Theta$ is globally generated by Corollary \eqref{basepoint}, the natural evaluation map
\begin{equation}
	\label{basepointmap}
H^0(\Theta) \rightarrow \Theta_{x} 
\end{equation}	
is surjective, where $\Theta_{x}$ is the fiber of $\Theta$ at $x$. Let $K = \{s \in H^0(\Theta) : s(x) =0\}$ be the kernel of the morphism \eqref{basepointmap}. Then $\text{dim}(K) = \text{dim}(H^0(\Theta)) -1$. We have $\text{dim}(H^0(\Theta)) >1$, otherwise the unique section $s$ has to be a non vanishing section which will induce $\Theta \cong \mathcal{O}_{SU_{X}(2)}$. This is a contradiction. Thus any section in $K$ will define an effective divisor $E$ containing $x$.\\

To prove $(2)$ it is enough to prove that the following is a strict inclusion
\[
\{s \in H^0(\Theta) : s(D) =0\} \subseteq \{s \in H^0(\Theta) : s(x) =0\} =K.
\] 
If it is an equality, we consider the following map defined by the linear system $\lvert \Theta \rvert$
\[
\phi : SU_{X}(2) \rightarrow \mathbb{P}(H^0(\Theta))
\]
\[
y \rightarrow [s_1(y), s_2(y), \cdots, s_{n-1}(y), s(y)],
\]
where $\{s_1, s_2, \cdots, s_{n-1}\}$ span the vector space $K$. The map $\phi_{\vert D}$ is the constant map $[0, 0, \cdots, 1]$. Hence $\Theta \cdot D = 0$ which contradicts the ampleness of $\Theta$.
\end{proof}
Let $V$ and $W$ be two subvarieties of complimentary dimension in a projective variety $Y$. We will assume $V$ and $W$ intersect properly. Let $P$ be a point of intersection. Then the intersection product $V \cdot W$ is equal to $\sum_{P \in V \cap W} m_P(V, W)$, where $m_P(V, W)$ is the intersection multiplicity at $P$. Let $m_P(V)$, and $m_P(W)$ be the multiplicity at the point $P$ of $V$, and $W$ respectively. Then in general one has the inequality $m_P(V, W) \ge m_P(V) \cdot m_P(W)$. The equality occurs if $\mathbb{T}C_{P}V \cap \mathbb{T}C_{P}W = \phi$ in $\mathbb{P}T_P Y$, where $\mathbb{T}C_{P}V$, and $\mathbb{T}C_{P}W$ are the projectivized tangent cones [\cite{MR3617981}, proposition 1.29., page 36]. Thus for any given point of intersection $P$ one has the following inequality
\begin{equation}
	\label{bezoutinequality}
	V \cdot W \ge m_P(V) \cdot m_P(W).
\end{equation}

$\textbf{Proof of Theorem 1.2}$ Let $x$ be a point in $SU_{X}(2)$. Let $D$ be a reduced and irreducible projective curve contained in $SU_{X}(2)$ which passes through $x$. Then, by Lemma \eqref{bezout} there exists a divisor $E$ in the linear system $\lvert\Theta \rvert$ satisfying $(1)$ and $(2)$ of Lemma \eqref{bezout}. Hence, by  \eqref{bezoutinequality} we will have
\[
D \cdot E \ge \text{mult}_x D \cdot \text{mult}_x E.
\]
 From \eqref{cantheta}, the canonical bundle $K_{SU_{X}(2)}$ satisfies the following relation with $\Theta$: %\;\cite[Theorem F.]{MR0999313}
\[
-K_{SU_{X}(2)} \cong \Theta^4
\]
Thus we have 
\[
\frac{D \cdot (-K_{SU_{X}(2)})}{\text{mult}_x D} = \frac{4 (D \cdot \Theta)}{\text{mult}_x D} = \frac{4 (D \cdot E)}{\text{mult}_x D} \ge 4 \cdot \text{mult}_x E \ge 4.
\]
\renewcommand{\abstractname}{Acknowledgements}
\begin{abstract}
	I am grateful to Prof. V. Balaji for posing this question to me, and to Prof. Krishna Hanumanthu for carefully reading this note and suggesting several improvements.
	%I am grateful to 
	%my advisor 
%%	Prof. V. Balaji for asking this question to me.
	% Additionally, the use of degeneration method in the proof of Proposition $2.2.$ was his suggestation. 
	%I take pleasure in thanking Prof. Krishna Hanumanthu for a careful reading of this note and for pointing out several pointing out several improvements.
	 % His comments have helped to improve the exposition. 
	 %Furthermore, I thank Dr. Sourav Das for his opinion on this note.
\end{abstract}
\bibliographystyle{alphaurl}
%\begin{thebibliography}{100}
\bibliography{A_note_on_a_theorem_of_Narasimhan_Ramanan.bib}	
	
%	\bibitem{B} T. Bauer, Seshadri constants and periods of polarized abelian varieties, 
%	Math. Ann. {\bf 312} (1998), no.~4, 607--623.

%\end{thebibliography}

\vspace{30mm}
\begin{flushleft}
{\scshape Tata Institute of Fundamental Research, Mumbai, Dr. Homi Bhabha Road, Navy Nagar, Colaba, Mumbai, Maharashtra 400005.}

{\fontfamily{cmtt}\selectfont
\textit{Email address :jagadish@math.tifr.res.in}}
\end{flushleft}
\end{document}